\documentclass[11pt]{article}
\usepackage{mathrsfs}
\usepackage{amsmath}
\usepackage{amsfonts}
\usepackage{amssymb}
\usepackage{latexsym}
\usepackage[dvips]{graphicx}
\newtheorem{theorem}{\bf Theorem}[section]
\newtheorem{lemma}[theorem]{\bf Lemma}
\newtheorem{proposition}[theorem]{\bf Proposition}

\newenvironment{proof}{\noindent{\em Proof:}}{\quad \hfill$\Box$\vspace{2ex}}

\def \bN {\Bbb N}
\def \bZ {\Bbb Z}

\def \bR {\Bbb R}

\def \bC {\Bbb C}

\def \bc {{\bf c}}

\def \cA {{\cal A}}
\def \cB {{\cal B}}

\def \cH {{\cal H}}

\def \cT {{\cal T}}
\def \cE {{\cal E}}

\def \cW {{\cal W}}

\def \and {\, \mbox{\rm and}\, }
\def \sinc {\,{\rm sinc}\,}

\def \span {\,{\rm span}\,}

\def \dist {\,{\rm dist}\,}

\makeatletter

\newcommand{\Rmnum}[1]{\expandafter\@slowromancap\romannumeral #1@}
\makeatother
\begin{document}
\title{\bf Exponential Approximation of Bandlimited Random Processes from Oversampling
 \thanks{Supported by Guangdong
Provincial Government of China through the ``Computational Science
Innovative Research Team" program.}}
\author{Wenjian Chen\thanks{School of Mathematics
and Computational Science, Sun Yat-sen University, Guangzhou 510275,
P. R. China. E-mail address: {\it wenjianchen66@gmail.com}.}\quad and \quad
Haizhang Zhang\thanks{School of Mathematics and Computational
Science and Guangdong Province Key Laboratory of Computational
Science,
 Sun Yat-sen University, Guangzhou 510275, P. R. China. E-mail address: {\it zhhaizh2@mail.sysu.edu.cn}. Supported in part by Natural Science Foundation of China under grants 11222103, 11101438 and 91130009, and by the US Army Research
 Office.}}
\date{}
\maketitle
\begin{abstract}
The Shannon sampling theorem for bandlimited wide sense stationary random processes was established in 1957, which and its extensions to various random processes have been widely studied since then. However, truncation of the Shannon series suffers the drawback of slow convergence. Specifically, it is well-known that the mean-square approximation error of the truncated series at $n$ points sampled at the exact Nyquist rate is of the order $O(\frac1{\sqrt{n}})$. We consider the reconstruction of bandlimited random processes from finite oversampling points, namely, the distance between consecutive points is smaller than the Nyquist sampling rate. The optimal deterministic linear reconstruction method and the associated intrinsic approximation error are studied. It is found that one can achieve exponentially-decaying (but not faster) approximation errors from oversampling. Two practical reconstruction methods with exponential approximation ability are also presented.

\noindent{\bf Keywords:} bandlimited random processes, oversampling, the intrinsic approximation error, exponential decayness, reproducing kernels

\noindent {\bf 2010 Mathematical Subject Classification: 62D05, 60G10}
\end{abstract}

\section{Introduction}
\setcounter{equation}{0}
The purpose of this paper is to establish exponentially approximating reconstruction methods from finite oversampling for bandlimited wide sense stationary random processes. The motivation comes from the slow convergence of the truncated Shannon series for random processes and the recent progress in the study of oversampling for bandlimited deterministic signals.

The foundational Shannon sampling theorem \cite{C. E. Shannon,E. T. Whittaker} states that every bandlimited deterministic signal can be completely recovered from its function values on sampling points equally-spaced at the exact Nyquist rate. We introduce the details for later discussion. Let $\cB_\delta$, $\delta>0$, denote the Paley-Wiener space of functions $f\in C(\bR)\cap L^2(\bR)$ with Fourier transform supported on $[-\delta,\delta]$. The Fourier transform is defined for $f\in L^1(\bR)$ by
$$
\hat{f}(\xi):=\frac1{\sqrt{2\pi}}\int_{\bR}f(t)e^{-it\xi}dt,\ \ \xi\in\bR,
$$
and for tempered distribution by a duality principle \cite{GW}. The Nyquist rate for the bandwidth $\delta$ is $\pi/\delta$. For simplicity, we discuss $\cB_\pi$ whose Nyquist rate is $1$. The Shannon sampling formula is
\begin{equation}\label{shannon}
f(t)=\sum_{j\in\bZ}f(j)\sinc(t-j),\ \ t\in\bR,\ f\in\cB_\pi,
\end{equation}
where the series converges absolutely and uniformly on $\bR$. Here, $\sinc$ denotes the sinc function
$$
\sinc(t):=\frac{\sin\pi t}{\pi t},\ \ t\in\bR.
$$

In practice, only finite and localized sampling data are available. For instance, to reconstruct the values of $f\in\cB_\pi$ on $[0,1]$, one may only have at hand the sample data $f(J_n)$. Here, $J_n:=\{j\in\bZ:-n+1\le j\le n\}$. In this case, it has long been observed that directly truncating the Shannon series (\ref{shannon}) suffers the drawback of slow convergence \cite{H. D. Helms and J. B. Thomas,D. Jagerman,Charles A. Micchelli}. Specifically, there exist some positive constants $C_1,C_2$ such that
$$
\frac{C_1}{\sqrt{n}}\le\sup\biggl\{\biggl|f(t)-\sum_{j\in J_n}f(j)\sinc(t-j)\biggr|:t\in[0,1],\ f\in\cB_\pi,\ \|f\|_{L^2(\bR)}\le 1\biggr\}\le \frac{C_2}{\sqrt{n}}.
$$

The Shannon sampling theorem and its extensions have been established for bandlimited wide sense stationary random processes \cite{Balakrishnan,Beutler,Brown,Houdre,Kolmogorov,Piranashvili,Seip,Song}. A random process $X = X(t, \omega)$ on
an underlying probability space $(\Omega, \mathscr{F},
P)$ is said to be {\it weak sense stationary (w.s.s.)}
if $X(t,\cdot)\in L^2(\Omega,dP(\omega))$ for each $t\in\bR$, that
is,
\begin{equation}\label{weak}
\parallel X(t,\cdot) \parallel_{L^2(\Omega,dP(\omega))} \ := \bigg(\int_{\Omega}\mid X(t, \omega) \mid^{2}
dP(\omega)\bigg)^{\frac{1}{2}} < +\infty, \ \ \ \ t\in \mathbb{R},
\end{equation}
and its autocorrelation function
\begin{equation}\label{RXfeature}
R_X(t,\tau)=\int_\Omega X(t,\omega)\overline{X(\tau,\omega)}dP(\omega),\ t,\tau\in\bR
\end{equation}
depends on $t-\tau$ only. It implies that
$$
R_X(t,\tau)=R_X(t-\tau,0).
     $$
We often abbreviate $R_X(t,0)$ as $R_{X}(t)$. A w.s.s. random
process $X(t, \omega)$ is said to be bandlimited if $R_X$ belongs to $\mathcal{B}_{\delta}$ for some $\delta>0$.

The Shannon sampling theorem for random processes, first proved in \cite{Balakrishnan}, states that for a w.s.s random process $X(t,\omega)$ with $R_X\in \cB_\pi$,
\begin{equation}\label{shannonrandom}
X(t,\omega)=\sum_{j\in\bZ}X(j,\omega)\sinc(t-j),\ \ t\in\bR.
\end{equation}
The above equation holds in the sense
\begin{equation}\label{meantrincation}
\lim_{n\to\infty}E\biggl|X(t,\omega)-\sum_{j\in J_n}X(j,\omega)\sinc(t-j)\biggr|^2=0,\ \ t\in\bR,
\end{equation}
where the expectation is taken with respect to the probability measure $P$. When using finite sampling points at the exact Nyquist rate, truncating the Shannon series (\ref{shannonrandom}) also results in a very slow reconstruction. The expectation in (\ref{meantrincation}) is of the order $O(\frac1{\sqrt{n}})$ as well, \cite{Beutler}.

For bandlimited deterministic signals, dramatic improvement of the approximation order appears suddenly if oversampling data are used. For simplicity, we let $f\in\cB_\delta$ with $\delta<\pi$. Thus, integer sampling points are equally-spaced by $1< \pi/\delta$ and hence constitute oversampling points. Modified Shannon series of the form
$$
\sum_{j\in J_n}f(j)\sinc(t-j)\phi(t-j)
$$
were proposed to reconstruct $f(t)$, $t\in[0,1]$ from its oversampling data $f(J_n)$. By setting
\begin{equation}\label{weightjagerman}
\phi(t):=\sinc^{m}(\frac{\pi - \delta}{\pi m}t), \ \ t\in\bR
\end{equation}
and optimizing about the parameter $m$, a reconstruction error bounded by
\begin{equation}\label{Jagermanresult}
\frac{C}{n(\pi-\delta)} \exp(-\frac{\pi - \delta}{e}n),\ \ n\ge \frac{e}{\pi-\delta}
\end{equation}
was established in \cite{D. Jagerman}. By choosing $\phi$ to be a Gaussian function with an optimal variance, references \cite{L. Qian 1,L. Qian 2,L. Qian 3} obtained an approximation error bounded by
$$
C\sqrt{\frac n{\pi-\delta}} \exp(-\frac{\pi - \delta}{2}n),\ \ n\ge \frac{8\pi}{(\pi-\delta)^3}.
$$
By optimizing the weight function $\phi$, it was proved in \cite{Charles A. Micchelli} that one can achieve the following upper bound for reconstruction errors
$$
\frac{C}{\sqrt{n}} \exp(-\frac{\pi - \delta}{2}n),\ \ n\in\bN.
$$
The optimal weight was later found to be a spline function in the Ph.D. thesis \cite{Haizhang Zhang}.

We believe that oversampling should always lead to exponential approximation for bandlimited signals. This paper aims to reveal this phenomenon for bandlimited random processes. We shall apply the modified Shannon series approach
\begin{equation}\label{modifiedshannon}
\sum_{j\in J_n}X(j,\omega)\sinc(t-j)\phi(t-j).
\end{equation}
However, directly applying the methods in \cite{D. Jagerman,Charles A. Micchelli,L. Qian 1} may not yield satisfactory error estimates. For example, we shall see in Section 3 that by directly using the result in \cite{D. Jagerman}, one will get an approximation error bounded by
$$
 C\sqrt{\frac{\ln
n}{n(\pi-\delta)}} \exp(-\frac{\pi - \delta}{2e}n).
$$
Compared to (\ref{Jagermanresult}), the dominating exponential part degenerates. To overcome this, careful analysis will be carried out in Section 3 to show that using the optimal spline weight function given in \cite{Haizhang Zhang} leads to an approximation error bounded by
$$
\sqrt{\frac{121}{200}}e^{3/4}
\sqrt{2+\pi-\delta}(\frac{\ln n+1}{2n})^{1/4}\exp\big(-\frac{\pi-\delta}2n\big).
$$

Before estimating the approximation error of the above two explicit and practical approximation methods, we shall discuss in Section 2 the optimal deterministic linear method for the reconstruction of a bandlimited w.s.s. random process from oversampling. We will see that the optimal linear method is impractical as it requires the solving of a highly ill-posed linear system. Our main purpose is to estimate the associated intrinsic approximation error, providing us guidelines and expectation in constructing practical reconstruction methods. This question is closely related to the optimal reconstruction in a reproducing kernel Hilbert space. Thanks to existing results on reproducing kernels, an upper and lower bound estimates will be established. In particular, the lower bound estimate reveals that there does not exist a deterministic linear reconstruction method with an approximation error tending to zero faster than exponential decayness.

\section{Optimal Linear Reconstruction and Intrinsic Errors}
\setcounter{equation}{0}

In this section, we consider the best possible linear method of reconstructing a w.s.s. random signal bandlimited to $[-\delta,\delta]$ with $\delta<\pi$ from its values on finite integer points. We shall find that this is closely related to the question of optimal reconstruction of functions in a reproducing kernel Hilbert space from its finite samples. For this reason, we first introduce the notion of reproducing kernels and reproducing kernel Hilbert spaces \cite{Aronszajn}.

A function $K:\bR\times\bR\to\bC$ is said to be {\it positive-definite} if for all finite distinct points $\cT:=\{t_j:1\le j\le n\}\subseteq \bR$, the $n\times n$ matrix
$$
K[\cT]:=[K(t_j,t_k):1\le j,k\le n]
$$
is hermitian and positive semi-definite. A positive-definite function $K$ on $\bR\times \bR$ corresponds uniquely to a reproducing kernel Hilbert space denoted as $\cH_K$ such that $K(t,\cdot)\in\cH_K$ for all $t\in \bR$ and
\begin{equation}\label{reproducing}
f(t)=\langle f,K(t,\cdot)\rangle_{\cH_K}\mbox{ for all }f\in\cH_K\mbox{ and }t\in\bR,
\end{equation}
where $\langle\cdot,\cdot\rangle$ denotes the inner product on $\cH_K$. By virtue of the above equation, $K$ is also called the {\it reproducing kernel} of $\cH_K$. Reproducing kernels are widely applicable to machine learning \cite{ScSm}. Besides positive-definiteness, another characterization \cite{CuckerSmale} of a reproducing kernel $K$ is the existence of a {\it feature map} $\Phi$ from $\bR$ to some Hilbert space $\cW$ such that
$$
K(t,\tau)=\langle\Phi(t),\Phi(\tau)\rangle_{\cW},\ \ t,\tau\in\bR.
$$

A typical class of reproducing kernel Hilbert spaces is the Paley-Wiener spaces $\cB_\delta$ of bandlimited functions. With the inner product
$$
\langle f, g \rangle_{\mathcal {B}_{\delta}}:= \langle f, g
\rangle_{L^{2}(\mathbb{R})} := \int_{\mathbb{R}}
f(x)\overline{g(x)}dx,
$$
$\mathcal {B}_{\delta}$ is a reproducing kernel
Hilbert space with the reproducing kernel
\begin{equation}\label{kdelta}
K_\delta(x, y) := \frac{\sin \delta(x - y)}{\pi(x - y)}, \ \ \ x, y\in
\mathbb{R}.
\end{equation}
Note that $K_\pi(x,y)=\sinc(x-y),\ x,y\in\bR$. The Shannon sampling theorem can be proved by noticing that
$$
{\{\sinc(\cdot - j): \ j\in \mathbb{Z}}\}
$$
is an orthonormal basis for $\mathcal {B}_{\pi}$. This basis together with the reproducing property (\ref{reproducing}) yields the following useful identity
\begin{equation}\label{following}
\sum_{j\in \mathbb{Z}} \mid f(j) \mid^{2} \ = \|f\|_{L^2(\bR)}^2,\ \ f\in \cB_\pi.
\end{equation}

Turning to the random case, we observe for a w.s.s. random process $X(t,\omega)$ that $R_X$ is a reproducing kernel on $\bR\times \bR$ because it has the feature map $\Phi(t):=X(t,\omega)\in L^2(\Omega,dP(\omega))$. Furthermore, it is translation-invariant in the sense that
$$
R_X(t+s,\tau+s)=R_X(t,\tau),\ \ s\in\bR.
$$
All continuous translation-invariant reproducing kernels $K$ on $\bR\times\bR$ are characterized by the Bochner theorem \cite{Bochner2} as
$$
K(t,\tau)=\int_{\bR}e^{i(t-\tau)\xi}d\mu(\xi),\ \ t,\tau\in\bR,
$$
where $\mu$ is a finite positive Borel measure on $\bR$. Since we are dealing with bandlimited random processes for which $R_X\in \cB_\delta$, there holds
\begin{equation}\label{rxfourier}
R_X(t,\tau)=\int_{-\delta}^\delta e^{i(t-\tau)\xi}\rho(\xi)d\xi,\ \ t,\tau\in\bR
\end{equation}
for some nonnegative function $\rho\in L^2([-\delta,\delta])$. We assume throughout the section that $\|\rho\|_{L^2([-\delta,\delta])}>0$ to avoid the trivial case.

For a general continuous translation-invariant reproducing kernel $K$ on $\bR\times\bR$, the optimal method of reconstructing functions $f$ in the unit ball of $\cH_K$ from their samples $f(J_n)$ on $J_n$ have been understood in \cite{Charles A. Micchelli}. To introduce the result, we denote for a Hilbert space $\cH$ by $U(\cH):=\{f\in\cH:\|f\|_{\cH}\le 1\}$ its closed unit ball, and by $\bC^{J_n}$ the set of all the functions on $J_n$. The purpose is to reconstruct the function values of $f\in\cH_K$ on $[0,1]$ from its samples $f|_{J_n}$. The maximum norm is engaged to measure the reconstruction error. Thus, a candidate reconstruction operator $\cA$ should map $f|_{J_n}$ to a function in $L^\infty([0,1])$. The approximation error of $\cA$ is measured by
$$
\cE_{n}(\cA,K):=\sup\{\|f-\cA(f|_{J_n})\|_{L^\infty([0,1])}:f\in U(\cH_K)\}.
$$
It is well-known \cite{optimal,Rivlin} that the optimal reconstruction operator is by the minimal norm interpolation whose approximation error attains the following intrinsic approximation error
$$
\cE_{n}(K):=\inf\{\cE_n(\cA,K):\mbox{ among all mappings }\cA:\, \bC^{J_n}\to L^\infty([0,1])\}.
$$
The intrinsic approximation error is characterized in \cite{Charles A. Micchelli} as
\begin{equation}\label{intrinsic}
\cE_{n}(K)=\sup\{\|f\|_{L^\infty([0,1])}:f\in U(\cH_K),\ f(J_n)=0\}.
\end{equation}

We consider in this section the optimal deterministic linear method of reconstructing a bandlimited random process $X(t,\omega)$ on $t\in[0,1]$ from its samples $X(j,\omega)$, $j\in J_n$. Specifically, we wish to find coefficient functions $c_j$'s on $[0,1]$ that minimize the quantity
\begin{equation}\label{distance}
\|X(t,\cdot)-\sum_{j\in
J_n}c_j(t)X(j,\cdot)\|_{L^2(\Omega,dP(\omega))}
\end{equation}
for each $t\in (0,1)$. The associated intrinsic approximation error is measured by
$$
\sup\{\|X(t,\cdot)-\sum_{j\in
J_n}c_j(t)X(j,\cdot)\|_{L^2(\Omega,dP(\omega))}:t\in(0,1)\}.
$$
We first make a simple observation about the optimal coefficients and the above intrinsic error. Let $V$ be a normed vector space, $x_0\in V$ and $A\subseteq V$. We shall denote by $\span A$ the linear subspace spanned by $A$, and denote by
$$
\dist_V(x_0,\span A)
$$
the distance of $x_0\in V$ to $\span A$ in $V$.

\begin{proposition}\label{randomopt}
The optimal coefficient functions $\bc=(c_j:j\in J_n)$ that minimize (\ref{distance}) are given by
\begin{equation}\label{optc}
\bc(t)=(R_X(t,j):j\in J_n)(R_X[J_n])^{-1},\ \ t\in(0,1)
\end{equation}
and the intrinsic error has the form
\begin{equation}\label{intrinsicform}
\sup\{\|X(t,\cdot)-\sum_{j\in
J_n}c_j(t)X(j,\cdot)\|_{L^2(\Omega,dP(\omega))}:t\in(0,1)\}=\cE_{n}(R_X).
\end{equation}
\end{proposition}
\begin{proof}
The minimum of the quantity (\ref{distance}) is the distance from $X(t,\cdot)$ to the linear space spanned by $X(j,\cdot)$, $j\in J_n$ in the Hilbert space $L^2(\Omega,dP(\omega))$. By the characterization of best approximation in Hilbert spaces, the optimal coefficients $c_j$'s are hence determined by the orthogonality conditions
$$
\langle X(t,\cdot)-\sum_{j\in J_n}c_j(t)X(j,\cdot),X(k,\cdot)\rangle_{L^2(\Omega,dP(\omega))}=0,\ \ k\in J_n.
$$
By (\ref{RXfeature}), we reach the linear system of equations
$$
\bc R_X[J_n]=(R_X(t,j):j\in J_n).
$$
To get (\ref{optc}), it remains to point out that $R_X(J_n)$ is nonsingular. We confirm this by showing that it is strictly positive-definite. Assume there are coefficients $\alpha_j\in\bC$, $j\in J_n$ such that
$$
\sum_{j\in J_n}\sum_{k\in J_n}\alpha_j\bar{\alpha_k}R_X(j,k)=0.
$$
It follows by (\ref{rxfourier}) that
$$
\int_{-\delta}^\delta \biggl|\sum_{j\in
J_n}\alpha_je^{ij\xi}\biggr|^2\rho(\xi)d\xi=0.
$$
Thus, the holomorphic function
$$
\varphi(z):=\sum_{j\in J_n}\alpha_je^{ijz},\ \ z\in\bC
$$
must vanish on the support of $\rho$, which is of positive Lebesgue measure. As a consequence, $\varphi$ vanishes everywhere on $\bC$, forcing $\alpha_j=0$ for each $j\in J_n$.

By choosing the optimal coefficients (\ref{optc}), we compute for each $t\in (0,1)$
$$
\|X(t,\cdot)-\sum_{j\in
J_n}c_j(t)X(j,\cdot)\|_{L^2(\Omega,dP(\omega))}={\bigg(R_X(0)-(R_X(t,j):j\in
J_n)(R_X[J_n])^{-1}(R_X(j,t):j\in J_n)^T}\bigg)^{\frac{1}{2}}.
$$
One observes that this is also equal to the distance
$\dist_{\cH_{R_X}}(R_X(t,\cdot),\span\{R_X(j,\cdot):j\in J_n\})$. To
complete the proof, it remains to show that it equals
$$
\sup\{|f(t)|:f\in U(\cH_{R_X}),\ f(J_n)=0\}.
$$
Let $f\in U(\cH_{R_X})$ with $f(J_n)=0$, $j\in J_n$. By (\ref{reproducing}),
$$
\langle f, R_X(j,\cdot)\rangle_{\cH_{R_X}}=f(j)=0,\ \ j\in J_n.
$$
Therefore, for arbitrary coefficients $\alpha_j\in \bC$, $j\in J_n$, there holds
$$
f(t)=\langle f, R_X(t,\cdot)\rangle_{\cH_{R_X}}=\langle f,
R_X(t,\cdot)-\sum_{j\in J_n}\alpha_j
R_X(j,\cdot)\rangle_{\cH_{R_X}}.
$$
By the Cauchy-Schwartz inequality, we get
$$
|f(t)|\le \|f\|_{\cH_{R_X}}\biggl\|R_X(t,\cdot)-\sum_{j\in
J_n}\alpha_j R_X(j,\cdot)\biggr\|_{\cH_{R_X}}\le
\biggl\|R_X(t,\cdot)-\sum_{j\in J_n}\alpha_j
R_X(j,\cdot)\biggr\|_{\cH_{R_X}}.
$$
Since this is true for arbitrary coefficients $\alpha_j$'s, we have
$$
\sup\{|f(t)|:f\in U(\cH_{R_X}),\ f(J_n)=0\}\le
\dist_{\cH_{R_X}}(R_X(t,\cdot),\span\{R_X(j,\cdot):j\in J_n\}).
$$
On the other hand, letting
$$
f=\frac{R_X(t,\cdot)-\sum_{j\in J_n}c_j(t)
R_X(j,\cdot)}{\biggl\|R_X(t,\cdot)-\sum_{j\in J_n}c_j(t)
R_X(j,\cdot)\biggr\|_{\cH_{R_X}}}
$$
shows that the equality holds. The proof is complete.
\end{proof}

We remark that although the optimal linear reconstruction algorithm is given in the above proposition, it is numerically intractable as the condition number of the kernel matrix $R_X[J_n]$ typically increases to infinity at an astonishing rate as $n$ increases. The following table exhibits such a phenomenon for the reproducing kernel $K_\delta$ defined by (\ref{kdelta}).

\begin{center}{\bf Table 3.1 The condition number of $K_\delta[J_n]$.}
$$
\begin{array}{c|ccccc}\hline\hline
          & n=1         &n=3      &n=5       &n=7&n=9\\\hline
          \delta=\frac{3\pi}4&1.86&15.5&248&4.98\times10^3&1.08\times10^5\\
            \delta=\frac{\pi}2&4.50&2.19\times10^3&1.84\times 10^6 &1.75\times10^9 &1.76\times10^{12}\\
          \delta=\frac{\pi}3&10.6&2.51\times10^5&7.06\times10^{9}&2.23\times10^{14}&5.13\times10^{16}\\\hline\hline
\end{array}
$$
\end{center}

Our purpose for this section is to estimate the intrinsic error $\cE_n(R_X)$ in order to give guidelines in constructing deterministic linear reconstruction methods. Two practical reconstruction algorithms with exponentially decaying approximation error will be given in the next section.

To estimate $\cE_n(R_X)$, we shall need two lemmas. The first one presents the lower and upper bound estimates for $\cE_n(K_\delta)$ established in \cite{Charles A. Micchelli}.

\begin{lemma}\label{estimatessinc} \cite{Charles A. Micchelli}
It holds for all $\delta\le\pi$ and $n\in\bN$ that
$$
\frac5{11e\sqrt{3\pi}} \frac{\delta}{2n+1}
\left(\frac\delta 4\right)^{2n}\le \cE_n(K_\delta)\le \sqrt{2+\pi-\delta}\frac{\sqrt{3}e}{\sqrt{2}\pi}\frac1{\sqrt{n}}\exp\left(-\frac{\pi-\delta}2n\right).
$$
\end{lemma}

The next lemma is about the inclusion relation of the reproducing kernel Hilbert spaces of translation-invariant reproducing kernels.
\begin{lemma}\label{inclusion} \cite{ZhangZhao}
Let $u,v$ be nonnegative functions in $L^1(\bR)$ and let $K,G$ be defined by
$$
K(t,\tau)=\int_{\bR}e^{i(t-\tau)\xi}u(\xi)d\xi,\ \ G(t,\tau)=\int_{\bR}e^{i(t-\tau)\xi}v(\xi)d\xi,\ \ t,\tau\in\bR.
$$
Then $\cH_K\subseteq\cH_G$ if and only if the set $\{t\in\bR:u(t)>0,\ v(t)=0\}$ has Lebesgue measure zero and the essential bound $\lambda$ of
$u/v$ on $\{t\in\bR:v(t)>0\}$ is finite, in which case there holds
$$
\|f\|_{\cH_G}\le \sqrt{\lambda}\|f\|_{\cH_K}.
$$
\end{lemma}

We are ready to present upper bound estimates for $\cE_n(R_X)$.

\begin{theorem}\label{upperbound}
Let $\delta\le\pi$ and suppose that the function $\rho$ in (\ref{rxfourier}) belongs to $L^\infty([-\delta,\delta])$. Then it holds for all $n\in\bN$ that
\begin{equation}\label{upperrx}
\cE_n(R_X)\le \|\sqrt{\rho}\|_{L^\infty([-\delta,\delta])}\sqrt{2+\pi-\delta}\frac{\sqrt{3}e}{\sqrt{\pi n}}\exp\left(-\frac{\pi-\delta}2n\right).
\end{equation}
\end{theorem}
\begin{proof}
Note that $K_\delta$ has the form
\begin{equation}\label{kdeltaform2}
K_\delta(t,\tau)=\frac1{2\pi}\int_{-\delta}^\delta e^{i(t-\tau)\xi}d\xi,\ \ t,\tau\in\bR.
\end{equation}
By equation (\ref{rxfourier}) and the assumption that $\rho\in L^\infty([-\delta,\delta])$, we get by Lemma \ref{inclusion} that $\cH_{R_X}\subseteq\cH_{K_\delta}$ and
$$
\|f\|_{\cH_{K_\delta}}\le \sqrt{2\pi}\|\sqrt{\rho}\|_{L^\infty([-\delta,\delta])}\|f\|_{\cH_{R_X}},\ \ f\in\cH_{R_X}.
$$
It follows that
$$
U(\cH_{R_X})\subseteq \sqrt{2\pi}\|\sqrt{\rho}\|_{L^\infty([-\delta,\delta])}U(\cH_{K_\delta}).
$$
This inclusion relation together with (\ref{intrinsic}) implies
$$
\cE_n(R_X)\le \sqrt{2\pi}\|\sqrt{\rho}\|_{L^\infty([-\delta,\delta])}\cE_n(K_\delta).
$$
Applying the upper bound estimate for $\cE_n(K_\delta)$ in Lemma \ref{estimatessinc} yields (\ref{upperrx}).
\end{proof}

We shall use a different estimate method for the general case when the function $\rho$ in (\ref{rxfourier}) only lies in $L^1([-\delta,\delta])$. For notational simplicity, we shall denote for each $t\in\bR$ by $e_t$ the complex exponential function
$$
e_t(\xi):=e^{it\xi},\ \ \xi\in \bR.
$$

\begin{lemma}\label{upperbound2lemma1}
Let $\rho\in L^1([-\delta,\delta])$ and $R_X$ be of the form (\ref{rxfourier}). It holds that
\begin{equation}\label{upperbound2lemma1eq}
\cE_n(R_X)\le \|\rho\|_{L^1([-\delta,\delta])}^{1/2}\sup\{\dist_{L^\infty([-\delta,\delta])}(e_t,\span\{e_j:j\in J_n\}):t\in(0,1)\}.
\end{equation}
\end{lemma}
\begin{proof}
By the proof Proposition \ref{randomopt},
$$
\cE_n(R_X)=\sup\biggl\{\dist_{\cH_{R_X}}(R_X(t,\cdot),\span\{R_X(j,\cdot):j\in J_n\}):\  t\in(0,1)\biggr\}.
$$
Direct computation also yields
$$
\dist_{\cH_{R_X}}(R_X(t,\cdot),\span\{R_X(j,\cdot):j\in J_n\})=\dist_{L^2([-\delta,\delta])}(\sqrt{\rho}e_t,\span\{\sqrt{\rho}e_j:j\in J_n\}),\ \ t\in(0,1).
$$
Noting that
$$
\dist_{L^2([-\delta,\delta])}(\sqrt{\rho}e_t,\span\{\sqrt{\rho}e_j:j\in J_n\})\le \|\rho\|_{L^1([-\delta,\delta])}^{1/2}\dist_{L^\infty([-\delta,\delta])}(e_t,\span\{e_j:j\in J_n\}),\ \ t\in(0,1)
$$
completes the proof.
\end{proof}

By Lemma \ref{upperbound2lemma1}, we shall estimate the distance of $e_t$, $t\in(0,1)$ to $\span\{e_j:j\in J_n\}$ in $L^\infty([-\delta,\delta])$. This will be done by truncating an exact series. Set for each $k\in\bN$ and $a\in\bR_+$ by $C^k_0([-a,a])$ the space of functions on $\bR$ with $k$ continuous derivatives that are supported on $[-a,a]$.

\begin{lemma}\label{upperbound2lemma2}
Set $k\in\bN$ and let $\phi\in C^k_0([-\pi+\delta,\pi-\delta])$ with $\hat{\phi}(0)=1$. Then it holds for each $t\in(0,1)$ that
\begin{equation}\label{upperbound2lemma2eq}
e_t(\xi)=\sum_{j\in\bZ}e_j(\xi)\sinc(t-j)\overline{\hat{\phi}(t-j)},\ \ \xi\in[-\delta,\delta],
\end{equation}
where the series converges absolutely for $\xi\in[-\delta,\delta]$.
\end{lemma}
\begin{proof}
We first note for $t\in(0,1)$ and $|j|\ge 2$ that
\begin{equation}\label{regularitydecayness}
|\hat{\phi}(t-j)|\le
\frac1{\sqrt{2\pi}}\|\phi^{(k)}\|_{L^1([-\pi+\delta,\pi-\delta])}|t-j|^{-k}.
\end{equation}
Thus, for each $t\in(0,1)$ the series
$$
\sum_{j\in\bZ}e_j\sinc(t-j)\hat{\phi}(t-j)
$$
defines a function in $L^2([-\delta,\delta])$. Taking the inner product of an arbitrary $f\in L^2([-\delta,\delta])$ with this function gives
$$
\sqrt{2\pi}\sum_{j\in\bZ}\hat{f}(j)\sinc(t-j)\hat{\phi}(t-j),
$$
which, by Lemma 4.1 in \cite{Charles A. Micchelli}, equals
$$
\sqrt{2\pi}\hat{f}(t)=\int_{-\delta}^\delta f(\xi)\overline{e_t(\xi)}d\xi.
$$
Therefore, it holds in $L^2([-\delta,\delta])$ that
$$
e_t=\sum_{j\in\bZ}e_j\sinc(t-j)\overline{\hat{\phi}(t-j)}.
$$
As a consequence, (\ref{upperbound2lemma2eq}) holds for almost every $\xi\in[-\delta,\delta]$. Since both sides are continuous on $[-\delta,\delta]$ by equation (\ref{regularitydecayness}), the identity holds for all $\xi\in[-\delta,\delta]$. That the series on the right hand side converges absolutely for each $\xi$ also follows from (\ref{regularitydecayness}).
\end{proof}

Let $\phi$ be as described in the above lemma. Making use of the identity (\ref{upperbound2lemma2eq}), we have
\begin{equation}\label{upperbound2eq3}
\dist_{L^\infty([-\delta,\delta])}(e_t,\span\{e_j:j\in J_n\})\le \sum_{j\notin J_n}|\sinc(t-j)||\hat{\phi}(t-j)|,\ \ t\in(0,1).
\end{equation}
We shall choose $\phi$ whose Fourier transform decays fast to zero at infinity. By (\ref{regularitydecayness}), one tends to choose $\phi$ with as much regularity as possible. However, as $k$ increases, the $L^1$ norm of $\phi^{(k)}$ might increase to infinity as well. The infimum of this crucial quantity is precisely estimated in \cite{Charles A. Micchelli}.
\begin{lemma}\label{upperbound2lemma3}
Set $k\in\bN$. It holds that
\begin{equation}\label{upperbound2lemma3eq}
\inf\left\{\frac1{\sqrt{2\pi}}\|\phi^{(k)}\|_{L^1([-\pi+\delta,\pi-\delta])}:\
\phi\in C^k_0([-\pi+\delta,\pi-\delta]),\ \hat \phi(0)=1\right\}=\frac{2^{k-1}}{(\pi-\delta)^k}k!.
\end{equation}
\end{lemma}

With the above three lemmas, we are ready to establish an upper bound estimate of $\cE_n(R_X)$ for the general case when $\rho\in L^1([-\delta,\delta])$.
\begin{theorem}\label{upperbound2}
Set $\delta<\pi$. Let $\rho\in L^1([-\delta,\delta])$ and $R_X$ be of the form (\ref{rxfourier}). It holds that
\begin{equation}\label{upperbound2eq}
\cE_n(R_X)\le \|\rho\|_{L^1([-\delta,\delta])}^{1/2}\frac{11e}{10\sqrt{\pi}}(\frac{\sqrt{\pi-\delta}}{2}+\frac{2\sqrt{2}}{\sqrt{\pi-\delta}})\frac1{\sqrt{n}}\exp(-\frac{\pi-\delta}2n),\quad n\ge \frac4{\pi-\delta}.
\end{equation}
\end{theorem}
\begin{proof}
Let $\phi\in C^k_0([-\pi+\delta,\pi-\delta])$ with $\hat{\phi}(0)=1$. By Lemmas \ref{upperbound2lemma1}, \ref{upperbound2lemma2}, and equation (\ref{upperbound2eq3}),
$$
\cE_n(R_X)\le \sup_{t\in(0,1)}\|\rho\|_{L^1([-\delta,\delta])}^{1/2}\sum_{j\notin J_n}|\sinc(t-j)||\hat{\phi}(t-j)|.
$$
Noticing (\ref{regularitydecayness}), we get for each $t\in(0,1)$
$$
\sum_{j\notin J_n}|\sinc(t-j)||\hat{\phi}(t-j)|\le \frac1{\pi\sqrt{2\pi}}\|\phi^{(k)}\|_{L^1([-\pi+\delta,\pi-\delta])}\sum_{j\notin J_n}\frac1{|t-j|^{k+1}}.
$$
Elementary analysis indicates that $\sum_{j\notin J_n}\frac1{|t-j|^{k+1}}$ attains its maximum on $[0,1]$ at $t=0$ or $t=1$. Thus, for all $t\in(0,1)$,
$$
\sum_{j\notin J_n}\frac1{|t-j|^{k+1}}\le \frac1{n^{k+1}}+2\sum_{j=n+1}^\infty \frac1{j^{k+1}}\le \frac1{n^{k+1}}+2\int_{n}^\infty \frac1{s^{k+1}}ds=(\frac1n+\frac2k)\frac1{n^k}.
$$
Combining the above three equations yields
$$
\cE_n(R_X)\le \|\rho\|_{L^1([-\delta,\delta])}^{1/2}\frac1{\pi\sqrt{2\pi}}\|\phi^{(k)}\|_{L^1([-\pi+\delta,\pi-\delta])}(\frac1n+\frac2k)\frac1{n^k}.
$$
We then take the infimum of the right hand side above with respect to all $\phi\in C^k_0([-\pi+\delta,\pi-\delta])$ with $\hat{\phi}(0)=1$. By Lemma \ref{upperbound2lemma3}, we have
$$
\cE_n(R_X)\le \|\rho\|_{L^1([-\delta,\delta])}^{1/2}\frac1{\pi}(\frac1n+\frac2k)\frac{2^{k-1}}{(\pi-\delta)^k}k!\frac1{n^k}.
$$
Recalling the Stirling forumula
\begin{equation}\label{stirling}
k!\le \frac{11\sqrt{2\pi}}{10}\left(\frac k e\right)^k\sqrt{k},\ \ k\in\bN,
\end{equation}
we arrive at
$$
\cE_n(R_X)\le \|\rho\|_{L^1([-\delta,\delta])}^{1/2}\frac{11\sqrt{2k}}{10\sqrt{\pi}}(\frac1{2n}+\frac1k)\left(\frac {2k}{ne(\pi-\delta)}\right)^k.
$$
If $n\ge 4/(\pi-\delta)$, we let $k=\lfloor \frac{n(\pi-\delta)}2\rfloor$ to get
$$
\cE_n(R_X)\le \|\rho\|_{L^1([-\delta,\delta])}^{1/2}\frac{11e}{10\sqrt{\pi}}(\frac{\sqrt{\pi-\delta}}{2\sqrt{n}}+\frac{2\sqrt{2}}{\sqrt{n(\pi-\delta)}})\exp(-\frac{\pi-\delta}2n),
$$
which yields (\ref{upperbound2eq}).
\end{proof}

The upper bound estimates (\ref{upperrx}) and (\ref{upperbound2eq}) indicate that for $\delta<\pi$, the intrinsic error $\cE_n(R_X)$ of the optimal linear reconstruction from oversampling decays exponentially. We shall show that under a mild assumption on $\rho$, it can not decrease to zero at a rate faster than exponential decayness. To this end, we need the following well-known structure of the reproducing kernel Hilbert space of a translation-invariant reproducing kernel (see, for example, \cite{Wendland}, page 139).

\begin{lemma}\label{translationstructure}
Let $\varphi$ be a nonnegative function in $L^1(\bR)$ and $K$ be the reproducing kernel defined by
$$
K(t,\tau):=\int_{\bR}e^{i(t-\tau)\xi}\varphi(\xi)d\xi,\ \ t,\tau\in \bR.
$$
Then $\cH_K$ consists of those functions $f\in C(\bR)$ whose Fourier transforms are contained in $L^1(\bR)$ and satisfy
$$
\|f\|_{\cH_K}:=\frac1{\sqrt{2\pi}}\biggl(\int_{\bR}\frac{|\hat{f}(-\xi)|^2}{\varphi(\xi)}d\xi\biggr)^{1/2}<+\infty.
$$
\end{lemma}

\begin{theorem}\label{lowerbound}
Suppose that there exists $[a,b]\subseteq[-\delta,\delta]$ such that
$$
\rho(\xi)\ge m,\ \ \xi\in[a,b]
$$
for some positive constant $m$. Let $\delta':=(b-a)/2$. Then it holds for all $n\in\bN$ that
\begin{equation}\label{lowerboundestimate}
\cE_n(R_X)\ge \frac{5\sqrt{2m}}{11e\sqrt{3}} \frac{\delta'}{2n+1}
\left(\frac{\delta'} 4\right)^{2n}.
\end{equation}
\end{theorem}
\begin{proof}
Set
$$
K(t,\tau):=\frac1{2\pi}\int_a^be^{i(t-\tau)\xi}d\xi,\ \ t,\tau\in\bR.
$$
By Lemma \ref{inclusion}, $\cH_K\subseteq \cH_{R_X}$ and for all $f\in\cH_K$
$$
\|f\|_{\cH_{R_X}}\le \frac1{\sqrt{2\pi m}}\|f\|_{\cH_K}.
$$
By similar arguments used in the proof of Theorem \ref{upperbound},
\begin{equation}\label{lowerboundeq1}
\cE_n(K)\le \frac1{\sqrt{2\pi m}}\cE_n(R_X).
\end{equation}
We then apply Lemma \ref{translationstructure} to $K$ and $K_{\delta'}$ to see that
$$
(Tf)(t):=e^{-i\frac{a+b}2t}f(t),\ \ t\in\bR
$$
is a linear isomorphism from $\cH_K$ to $\cH_{K_{\delta'}}$. Since
$$
\|Tf\|_{L^\infty((0,1))}=\|f\|_{L^\infty((0,1))}
$$
and $(Tf)(J_n)=0$ if and only if $f(J_n)=0$, we observe from (\ref{intrinsic}) that
\begin{equation}\label{lowerboundeq2}
\cE_n(K)=\cE_n(K_{\delta'}).
\end{equation}
Combining (\ref{lowerboundeq1}), (\ref{lowerboundeq2}) and the lower bound estimate for $\cE_n(K_{\delta'})$ in Lemma \ref{estimatessinc} yields (\ref{lowerboundestimate}).
\end{proof}
\section{Practical Reconstruction with Exponentially Decaying Errors}
\setcounter{equation}{0}

In this section, we present two practical reconstruction methods with exponentially decaying approximation errors for bandlimited random processes. We shall consider the modified Shannon series approach (\ref{modifiedshannon}). In the first method, we shall use the multiplier proposed in \cite{D. Jagerman}
$$
g_m(t) := \sinc(t)\sinc^{m}(\frac{\pi - \delta}{\pi m}t), \ \
t\in \mathbb{R}, \ m\in \mathbb{N}
$$
and shall directly apply the following error estimate established therein.

\begin{lemma}\label{lemma2}
There exists a positive constant $C$ such that for all $n \geq e / (\pi - \delta)$ and $m = \lfloor n (\pi - \delta)/e \rfloor$
\begin{equation}\label{jagermanestimate}
\sup{\{\mid f(t)-\sum_{j\in J_n}f(j)g_m(t - j)\mid: t\in (0,1), f
\in U(\mathcal {B}_{\delta})}\} \ \leq \ \frac{C}{n(\pi -
\delta)}\exp(-\frac{\pi - \delta}{e}n).
\end{equation}
\end{lemma}

Our first reconstruction algorithm is explicitly defined by
$$
(\mathcal {A}_{1}X)(t,\omega):= \sum_{j\in J_n} X(j, \omega) g_m(t - j)
$$
where $m = \lfloor n (\pi - \delta)/e \rfloor$. The overall
approximation error of $\mathcal
{A}_{1}$ is measured by
$$
\mathcal{E}_{n}(\mathcal {A}_{1}) := \sup{\big\{\parallel X(t,
\cdot)-(\mathcal {A}_{1}X)(t,\cdot)
\parallel_{L^{2}(\Omega, dP(\omega))}:t\in (0,1), R_{X} \in U(\mathcal {B}_{\delta})}\big\}.
$$
We show below that $\mathcal{E}_{n}(\mathcal {A}_{1})$ decays
exponentially to zero as $n$ tends to infinity.

\begin{theorem}\label{practical1}
There exists a positive constant $C$ such that for all $n \geq
\max{\{e / (\pi - \delta), e^{2}}\}$,
\begin{equation}\label{practical1eq}
\mathcal{E}_{n}(\mathcal {A}_{1}) \ \leq \ C\sqrt{\frac{\ln
n}{n(\pi-\delta)}}  \exp(-\frac{\pi - \delta}{2e}n).
\end{equation}
\end{theorem}
\begin{proof}
We first compute by (\ref{RXfeature}) to get
\begin{eqnarray*}
&&\parallel X(t, \omega)-(\mathcal {A}_{1}X)(t,\omega)
\parallel^{2}_{L^{2}(\Omega, dP(\omega))}\\
&=& R_{X}(0)-\sum_{j\in J_n}R_{X}(t - j)g_m(t - j)-\sum_{j\in J_n}
\big(R_{X}(j-t)-\sum_{k\in J_n}R_{X}(j - k)g_m(t-k)\big)g_m(t - j).
\end{eqnarray*}
Let $R_X\in U(\cB_\delta)$. Then $R_{X}(u - \cdot)\in U(\cB_\delta)$ for each $u \in \mathbb{R}$. We apply Lemma
\ref{lemma2} to $R_X(t-\cdot)$ to get for $n \geq e / (\pi -
\delta)$ that
\begin{equation}\label{we}
\sup{\big\{\mid R_{X}(0)-\sum_{j\in J_n}R_{X}(t - j)g_m(t - j)\mid :
t \in (0, 1)
}\big\}\leq\frac{C}{n(\pi - \delta)}\exp(-\frac{\pi - \delta}{e}n).
\end{equation}
We then apply Lemma \ref{lemma2} to each $R_{X}( j-\cdot)$ for $j\in\bZ$ to have
$$
\sup{\big\{\mid R_{X}(j-t) - \sum_{k\in J_n}R_{X}(j-k) g_m(t -
k)\mid : t\in (0, 1)
}\big\}\leq\frac{C}{n(\pi - \delta)}\exp(-\frac{\pi - \delta}{e}n).
$$
Therefore, we attain for each $t\in(0,1)$ that
\begin{equation}\label{attain}
\biggl|\sum_{j\in J_n} \big(R_{X}(j-t) -
\sum_{k\in J_n}R_{X}(j-k) g_m(t - k)\big)g_m(t - j)\biggr| \leq\frac{C}{n(\pi - \delta)}\exp(-\frac{\pi - \delta}{e}n)
\sum_{j\in J_n} \mid g_m(t - j)\mid.
\end{equation}
Note that the $\sinc$ function is uniformly bounded by $1$. By this fact, when $t\in(0,1)$ and $n\ge e^2$,
\begin{equation}\label{Choosing}
\sum_{j\in J_n}|g_{m}(t - j)|\le
\sum_{j\in J_n} |\sinc(t - j)| \ \leq \ C'\ln n,
\end{equation}
where $C'$ is a positive constant. Combining (\ref{we}), (\ref{attain}), and (\ref{Choosing}) completes the proof.
\end{proof}

One sees that the dominating exponential part in the estimate (\ref{practical1eq}) degenerates compared to the estimate (\ref{jagermanestimate}) for deterministic signals. Through more careful analysis, we shall show that one is able to attain the same optimal order of exponential decayness for bandlimited random processes as for deterministic signals. Moreover, in order to present an explicit and practical reconstruction algorithm, we desire a function that achieves the infimum (\ref{upperbound2lemma3eq}).

It has been shown that there does not exist a function $\phi\in C^{(k)}_0([-\pi+\delta,\pi-\delta])$ that attains the infimum (\ref{upperbound2lemma3eq}). However, a spline function $B_k$ belonging to $C^{(k-2)}_0([-\pi+\delta,\pi-\delta])$ that satisfies
\begin{equation}\label{keyspline1}
|\hat{B_k}(\xi)|\le \frac{2^{k-1}k!}{(\pi-\delta)^k|\xi|^k},\ \ \xi\ne0
\end{equation}
and
\begin{equation}\label{keyspline2}
f(t)=\sum_{j\in\bZ}f(j)\sinc(t-j)\hat{B_k}(t-j),\ \ \mbox{for all } t\in\bR\mbox{ and }f\in\cB_\delta
\end{equation}
is constructed in \cite{Haizhang Zhang}. The spline function is determined by
$$
\hat{B_k}(\xi)=\frac1{\sqrt{2\pi}(i(\pi-\delta)\xi)^k}\sum_{j=0}^k(-1)^{k-j}\alpha_{kj}e^{-ix_{kj}(\pi-\delta)\xi},\
\ \xi\in\bR,
$$
where
$$
x_{kj}:=\cos \frac{j\pi}k,\ \ 0\le j\le k
$$
are the zeros of the first-kind Chebyshev polynomial of
degree $k$ on $[-1,1]$, and $\alpha_{kj}$, $0\le j\le k$ are the unique nonnegative constants satisfying
$$
\sum_{j=0}^k(-1)^{k-j}\alpha_{kj}x_{kj}^l=0,\ \ 0\le l\le k-1
$$
and
$$
\sum_{j=0}^k\alpha_{kj}=\sqrt{2\pi}2^{k-1}k!.
$$

Our second reconstruction algorithm is explicitly given by
$$
(\mathcal {A}_{2}X)(t,\omega):= \sum_{j\in J_n} X(j,\omega) \sinc(t
- j)\hat{B_m}(t - j),
$$
where $m= \lfloor n(\pi-\delta)/2 \rfloor$. We shall need one last lemma to estimate the associated approximation error
\begin{equation}
\mathcal{E}_{n}(\mathcal {A}_{2}) := \sup{\big\{\parallel X(t,
\cdot)-(\mathcal {A}_{2}X)(t,\cdot)
\parallel_{L^{2}(\Omega, dP(\omega))}: t\in (0,1), R_{X} \in U(\mathcal {B}_{\delta})}\big\}.
\end{equation}

\begin{lemma}\label{6}
\cite{W. Splettst} Let $q>1$ and $1/p+1/q=1$. It holds for all $t\in\mathbb{R}$ that
$$
\sum_{j\in\bZ}\mid \sinc(t-j)\mid^{q} < p.
$$
\end{lemma}


\begin{theorem}
For each $n> \max{\{2/(\pi-\delta), e}\}$, there holds
\begin{equation}\label{integration}
\mathcal{E}_{n}(\mathcal {A}_{2}) \ \leq \sqrt{\frac{121}{200}}e^{3/4}
\sqrt{2+\pi-\delta}(\frac{\ln n+1}{2n})^{1/4}\exp\big(-\frac{\pi-\delta}2n\big).
\end{equation}
\end{theorem}
\begin{proof}
Let $R_X\in U(\cB_\delta)$, $t\in(0,1)$, and $h:=\sinc\,\hat{\cB_m}$. We compute by (\ref{RXfeature}) that
\begin{equation}\label{get}
\begin{array}{rl}
&\displaystyle{\| X(t, \cdot)-(\mathcal {A}_{2}X)(t,\cdot)
\|^{2}_{L^{2}(\Omega, dP(\omega))}}\\
\displaystyle{=}&\displaystyle{R_{X}(0) - \sum_{j\in J_{n}}R_{X}(t -
j) h(t - j) - \sum_{k\in J_{n}} \big(R_{X}(k-t) - \sum_{j\in
J_{n}}R_{X}(k - j) h(t - j)\big) h(t - k)}.
\end{array}
\end{equation}
Introduce
$$
F(u):=R_{X}(u-t)-\sum_{j\in J_{n}}R_{X}(u-j)h(t-j), \quad
u\in \mathbb{R}.
$$
Clearly, $F\in\mathcal {B}_{\delta}$. By (\ref{keyspline2}),
$$
\mid F(u) - \sum_{k\in J_{n}} F(k)h(u - k) \mid \ =\
\mid \sum_{k\notin J_{n}}F(k)h(u - k) \mid, \ \ u\in\bR.
$$
Letting $u = t$ in the above equations yields from (\ref{get}) that
$$
\| X(t, \cdot)-(\mathcal {A}_{2}X)(t,\cdot)
\|^{2}_{L^{2}(\Omega, dP(\omega))}=\sum_{k\notin J_n}\big(R_{X}(k-t) - \sum_{j\in
J_{n}}R_{X}(k - j) h(t - j)\big) h(t - k).
$$
Since $R_{X}(k - \cdot)\in U(\mathcal
{B}_{\delta})$ for each $k\in\bZ$, we get by (\ref{keyspline2})
\begin{equation}\label{each}
\parallel X(t, \cdot)-(\mathcal {A}_{2}X)(t,\cdot)
\parallel^{2}_{L^{2}(\Omega, dP(\omega))} \ = \ \mid \sum_{k\notin J_{n}}\bigg(\sum_{j\notin J_{n}}R_{X}(k - j)
h(t - j)\bigg)h(t - k)\mid.
\end{equation}

We then choose $\gamma\in(1/2,1)$ and apply the Cauchy-Schwarz inequality
twice to attain
{\small
\begin{equation}\label{twice}
\begin{array}{rl}
&\displaystyle{\mid \sum_{k\notin J_{n}}\big(\sum_{j\notin
J_{n}}R_{X}(k - j)
h(t - j)\big)h(t - k)\mid^{2}}\\
\displaystyle{\leq}&\displaystyle{\biggl[\sum_{k\notin
J_{n}}(\sum_{j\notin J_{n}}\mid R_{X}(k - j) \sinc^{\gamma}(t -
j)\mid^{2})\biggr]\biggl[\sum_{j\notin J_{n}}\mid \sinc^{1 - \gamma}(t -
j)\hat{B_m}(t-j)\mid^{2}\biggr]\biggl[\sum_{k\notin J_{n}}\mid h(t - k)\mid^{2}\biggr]}.
\end{array}
\end{equation}
}
It remains to estimate the three sums above. First, since
$R_{X}(\cdot - j)\in U(\mathcal {B}_{\delta})\subseteq U(\mathcal {B}_{\pi})$, by (\ref{following})
\begin{equation}\label{formula}
\sum_{k\in\bZ} \mid R_{X}(k - j)\mid^{2} \ = \
\int_{\mathbb{R}}\mid R_{X}(u - j)\mid^{2} du=\|R_X\|_{\cB_\pi}^2 \ \leq \ 1, \ \ j\in
\mathbb{Z}.
\end{equation}
By (\ref{formula}) and Lemma \ref{6}, we get by exchanging the order of summation that
\begin{equation}\label{and}
\sum_{k\notin J_{n}}\big(\sum_{j\notin J_{n}}\mid R_{X}(k - j)
\sinc^{\gamma}(t - j)\mid^{2}\big) \leq \sum_{j\notin J_{n}}\mid
\sinc^{2\gamma}(t - j)\mid<\frac{2\gamma}{2\gamma - 1}
\end{equation}
Next, it follows from (\ref{keyspline1}) that
\begin{equation}\label{deduce}
\sum_{j\notin J_{n}}\mid \sinc^{1 - \gamma}(t - j)\hat{B_m}(t -
j)\mid^{2} \ \leq \ \bigg(\frac{1}{\pi^{1 -
\gamma}}\frac{2^{m-1}m!}{(\pi-\delta)^{m}}\bigg)^{2}\sum_{j\notin
J_{n}}\frac{1}{(t - j)^{2m + 2(1 - \gamma)}}.
\end{equation}
Again, by (\ref{keyspline1}),
\begin{equation}\label{have}
\sum_{k\notin J_{n}}\mid h(t - k)\mid^{2} \ \leq \
\bigg(\frac{1}{\pi}\frac{2^{m-1}m!}{(\pi-\delta)^{m}}\bigg)^{2}\sum_{k\notin
J_{n}}\frac{1}{(t - k)^{2m + 2}}.
\end{equation}
Combining (\ref{each}), (\ref{twice}), (\ref{and}), (\ref{deduce}),
and (\ref{have}) yields
\begin{equation}\label{By}
\begin{array}{rl}
\displaystyle{(\mathcal{E}_{n}(\mathcal {A}_{2}))^2}\ \
\displaystyle{\leq}&\displaystyle{\bigg(\frac{2\gamma}{2\gamma -
1}\bigg)^{\frac{1}{2}} \frac{1}{\pi^{2-\gamma}}\bigg(\frac{2^{m-1}m!}{(\pi-\delta)^{m}}\bigg)^{2}\sup{\{\sqrt{h_{1}(t)h_2(t)}: t\in (0,1)\}}},
\end{array}
\end{equation}
where
$$
h_{1}(t): =\sum_{j\notin J_{n}}\frac{1}{(t - j)^{2m + 2(1 -
\gamma)}}, \ \ \ t\in (0,1),
$$
and
$$
h_{2}(t): =\sum_{j\notin J_{n}}\frac{1}{(t - j)^{2m + 2}}, \ \ \
t\in (0,1).
$$
By a similar estimate technique as that in the proof of Theorem \ref{upperbound2},
$$
h_{1}(t) \leq \frac{1}{n^{2m +2-
2\gamma}}+\sum_{j=n+1}^{+\infty}\frac{2}{j^{2m +2-2\gamma}}\leq
\frac{1}{n^{2m+1-2\gamma}}(\frac{1}{n}+\frac{1}{m+\frac12-\gamma}),\ \ t\in(0,1),
$$
which combined with the Stirling formula (\ref{stirling}) gives
\begin{equation}\label{Connecting}
\frac{2^{m-1}m!}{(\pi-\delta)^{m}} \ \sup{\{\sqrt{h_{1}(t)}: \ t\in
(0,1)}\} \ \leq \ \frac{11\sqrt{2\pi}}{20}\
n^{\frac{2\gamma-1}{2}}\ \sqrt{\frac{m}{n}+\frac m{m+\frac12-\gamma}}\
\bigg(\frac{2m}{(\pi-\delta)en}\bigg)^{m}.
\end{equation}
Likewise,
\begin{equation}\label{as}
\frac{1}{\pi^{2-\gamma}}\frac{2^{m-1}m!}{(\pi-\delta)^{m}} \
\sup{\{\sqrt{h_{2}(t)}: \ t\in (0,1)}\} \ \leq \
\frac{11\sqrt{2}}{20\sqrt{\pi n}}\
\sqrt{\frac{m}{n}+1}\bigg(\frac{2m}{(\pi-\delta)en}\bigg)^{m}.
\end{equation}

We observe from (\ref{By}), (\ref{Connecting}), and (\ref{as})
$$
(\mathcal{E}_{n}(\mathcal {A}_{2}))^2 \leq \bigg(\frac{2\gamma}{2\gamma
- 1} n^{2\gamma}\bigg)^{\frac{1}{2}} \frac{121}{100\sqrt{2}}\frac{1}{n}
\bigg(\frac{m}{n}+1\bigg)\bigg(\frac{2m}{(\pi-\delta)en}\bigg)^{2m}.
$$
We optimize the right hand side above with respect to $\gamma$ to choose $\gamma = \frac{1}{2}(1+\frac{1}{\ln n})$ when $n> e$. With this choice,
$$
(\mathcal{E}_{n}(\mathcal {A}_{2}))^2 \leq \frac{121\sqrt{e}}{100\sqrt{2}}
\sqrt{\frac{\ln
n+1}{n}}\bigg(\frac{m}{n}+1\bigg)\bigg(\frac{2m}{(\pi-\delta)en}\bigg)^{2m}.
$$
For $n>\max{\{2/(\pi-\delta),e}\}$, we substitute $m=\lfloor
n(\pi-\delta)/2\rfloor$ into the right hand side above
to obtain (\ref{integration}) and complete
the proof.
\end{proof}

We remark that estimate (\ref{integration}) has the same order of exponential decayness as that obtained in \cite{Charles A. Micchelli} for bandlimited deterministic signals.

\end{document}